\documentclass[12pt]{amsart}  

\usepackage{custom_preamble} 

\begin{document}

\title{On monochromatic solutions to $x-y=z^2$}

\author{\tsname}
\address{\tsaddress}
\email{\tsemail}

\begin{abstract}
For $k \in \N$, write $S(k)$ for the largest natural number such that there is a $k$-colouring of $\{1,\dots,S(k)\}$ with no monochromatic solution to $x-y=z^2$.  That $S(k)$ exists is a result of Bergelson, and a simple example shows that $S(k)\geq 2^{2^{k-1}}$.  The purpose of this note is to show that $S(k)\leq 2^{2^{2^{O(k)}}}$.
\end{abstract}

\maketitle

\begin{center}
\emph{To Endre Szemer{\'e}di on his 80th birthday}
\end{center}

\section{Introduction}

In \cite{khasze::}, Khalfalah and Szemer{\'e}di answered a question of Roth, Erd{\H o}s, S{\'a}rk{\" o}zy, and S{\'o}s by showing that for $r \in \N$ and $N$ sufficiently large in terms of $r$, any $r$-colouring of $[N]:=\{1,\dots,N\}$ contains two distinct elements $x$ and $y$ with the same colour and $x+y=z^2$ for some natural $z$.

On the other hand it was shown by Csikv{\'a}ri, Gyarmati, and S{\'a}rk{\" o}zy in \cite[Theorem 3]{csigyasar::0} that one cannot extend the Khalfalah-Szemer{\'e}di result to ask that $z$ be the same colour as $x$ and $y$.  This was refined by Green and Lindqvist \cite{grelin::} who showed that there are $3$-colourings of $\N$ without solutions to $x+y=z^2$ with $x$ and $y$ distinct and all of $x$, $y$, and $z$ having the same colour; and furthermore they showed that $3$ cannot be reduced to $2$.

If $x+y$ is replaced by $x-y$ things are quite different.  Here the celebrated Furstenberg-S{\'a}rk{\" o}zy Theorem (so named in view of \cite{sar::0}, and the comment after the main theorem there) gives a density analogue of the Khalfalah-Szemer{\'e}di result meaning that for any $\delta \in (0,1]$ if $N$ is sufficiently large in terms of $\delta$ then any subset of $[N]$ of size at least $\delta N$ contains two distinct elements $x$ and $y$ such that $x-y=z^2$ for some $z \in \N$.

In the Furstenberg-S{\'a}rk{\" o}zy Theorem we cannot ask for $z$ in the same set as $x$ and $y$ -- consider $\{x \in [N]: x \equiv 1 \pmod 3\}$ -- however Bergelson \cite[p53]{ber::} using a method from \cite{ber::2} showed that in any $k$-colouring of $\N$ there are solutions to $x-y=z^2$ with $x$, $y$ and $z$ all the same colour.  Our purpose is to establish the following quantitative result.
\begin{theorem}\label{thm.mn}
Suppose that $k,N \in \N$ are such that there is a $k$-colouring of $\{1,\dots,N\}$ with no monochromatic solutions to $x-y=z^2$.  Then $N\leq 2^{2^{2^{O(k)}}}$.
\end{theorem}
This is not the first quantitative result in this direction, in fact Lindqvist gave a bound in \cite[Theorem 5.1.2]{lin::3}. 

This bound cannot be replaced by anything smaller than $2^{2^{k-1}}$: if $N=2^{2^{k-1}}$, consider the colouring with colour classes
\begin{equation*}
\mathcal{C}:=\{\{1\}\}\cup\{\{2^{2^i},\dots,2^{2^{i+1}}\}: 0 \leq i \leq k-2\}.
\end{equation*}
This is a $k$-colouring and if $x,y,z \in \{2^{2^i},\dots,2^{2^{i+1}}\}$ then $x-y<2^{2^{i+1}} \leq (2^{2^i})^2 \leq z^2$ and so $x-y \neq z^2$, and if $x,y,z \in \{1\}$ then $x-y=0<1=z^2$.  It follows that this colouring contains no monochromatic solutions to $x-y=z^2$.

Prendiville in \cite[Theorem 1.2]{pre::1} as a special case of a much more general result has established a counting version of Theorem \ref{thm.mn} showing that there is a colour class with $\Omega_k(N^{3/2^k})$ solutions to $x-y=z^2$ (at least once $k$ is sufficiently large); the above construction show that this is close to optimal.

Finally, we remark that it is a well-known open problem to ask for monochromatic solutions to $x^2-y^2=z^2$ in place of $x-y=z^2$ (see \emph{e.g.} \cite[Problem 3.9]{crolev::} ), and some modular analogues have been investigated by Lindqvist \cite{lin::2}, including the modular version of Theorem \ref{thm.mn}.

\subsection*{Notation} When writing $A \subset B$ we do \emph{not} require the inclusion to be strict, and by $O(1)$ we mean an absolute constant.

\section{The argument}

The argument is an iterative application of the following.
\begin{corollary}\label{cor.ct}
Suppose that $A \subset [N]$ has size at least $\alpha N$.  Then either $N \leq \exp(\alpha^{-O(1)})$ or there are natural numbers $r \leq \exp(\alpha^{-O(1)})$ and $L \geq N^{\frac{1}{4}}$  such that
\begin{equation*}
\#\{x \in r \cdot [L]: x^2 \in A-A\} \geq \frac{1}{2}\alpha L.
\end{equation*}
\end{corollary}
This corollary is proved in \S\ref{sec.fs}, but reading it out of S{\'a}rk{\" o}zy's work \cite{sar::0} is completely routine.

\begin{proof}[Proof of Theorem \ref{thm.mn}]
Let $\mathcal{C}$ be a cover of $[N]$ of size $k$. We proceed iteratively: at stage $i$ we have natural numbers $N_i$ and $d_i$, a set $J_i \subset \mathcal{C}$, and integers $x_{C,i}$ for each $C \in J_i$, and a final integer $x_i$.  We write
\begin{equation*}
S_i:=\bigcap_{C \in J_i}{(x_{C,i}+C)} \text{ and } \alpha_i:=\frac{\#(x_i+S_i)\cap d_i^2\cdot [N_i]}{N_i}.
\end{equation*}
Let $A:=\{x \in [N_i]:d_i^2x \in x_i+S_i\}$ so that $A \subset [N_i]$ and $\#A \geq \alpha_iN_i$.  By Corollary \ref{cor.ct} applied to $A$ either $N_i \leq \exp(\alpha_i^{-O(1)})$ and we terminate, or there are natural numbers $L_i$ and $r_i$ with
\begin{equation*}
L_i \geq N_i^{\frac{1}{4}} \text{ and } r_i \leq \exp(\alpha_i^{-O(1)}) 
\end{equation*}
such that
\begin{equation*}
\#\{x \in r_i\cdot [L_i] : x^2 \in S_i-S_i\} \geq \frac{1}{2}\alpha_i L_{i}.
\end{equation*}
But then by design $x^2 \in A-A$ and hence
\begin{equation*}
(d_ix)^2 \in d_i^2A - d_i^2A \subset (x_i+S_i)-(x_i+S_i) \subset \bigcap_{C \in J_i}{(C-C)},
\end{equation*}
and so putting $d_{i+1}:=d_ir_i$ we have
\begin{equation}\label{eqn.5}
\#\left\{x \in d_{i+1}\cdot [L_i] : x^2 \in \bigcap_{C \in J_i}{(C-C)}\right\} \geq \frac{1}{2}\alpha_i L_i.
\end{equation}
It follows that assuming we have no monochromatic triple, the set on the left of (\ref{eqn.5}) must be covered by the sets in $\mathcal{C} \setminus J_i$.  By averaging we conclude that there is some $C_i \in \mathcal{C} \setminus J_i$ such that
\begin{equation*}
\#(d_{i+1}\cdot [L_i]) \cap C_i \geq \frac{1}{2k}\alpha_i L_i,
\end{equation*}
and so putting $N_{i+1}:=L_i/r_id_{i+1}$ there is some $x_*$ such that
\begin{equation*}
\#(x_*+d_{i+1}^2\cdot [N_{i+1}]) \cap C_i \geq \frac{1}{2k}\alpha_i N_{i+1}.
\end{equation*}
Finally, since $r_iN_{i+1} \leq N_i$ we have
\begin{equation*}
\#d_i^2\cdot [N_i] - d_{i+1}^2\cdot [N_{i+1}] = \#[N_i]-r_i^2\cdot [N_{i+1}] \leq 2N_i,
\end{equation*}
whence
\begin{align}
\nonumber \frac{1}{2k}\alpha_i^2 N_iN_{i+1} & \leq\#(x_i+S_i)\cap (d_i^2\cdot [N_i])\#(C_i -x_*)\cap (d_{i+1}^2\cdot [N_{i+1}])\\
\nonumber & =  \sum_x{\#(x+((x_i+S_i)\cap (d_i^2\cdot [N_i]))) \cap ((C_i -x_*)\cap (d_{i+1}^2\cdot [N_{i+1}]))}\\
\nonumber & \leq \max_x{\#(x+x_i+S_i) \cap (C_i-x_*) \cap (d_{i+1}^2\cdot [N_{i+1}]))}\\
\nonumber & \qquad\qquad\qquad\qquad\qquad\qquad\qquad\times \#d_i^2\cdot [N_i] - d_{i+1}^2\cdot [N_{i+1}]\\
\label{eqn.ig}& \leq 2N_i\max_x{\#(x+S_i) \cap (C_i-x_*) \cap (d_{i+1}\cdot [N_{i+1}]))}.
\end{align}
Let $x_{i+1}$ be an integer such that maximum in (\ref{eqn.ig}) on the right is achieved.  Put $x_{C_i,i}:=-x_*$, $x_{C,i+1}=x_{i+1}+x_{C,i}$, $J_{i+1}:=J_i \cup \{C_i\}$, and note that
\begin{equation*}
\alpha_{i+1} \geq \frac{1}{4k}\alpha_i^2.
\end{equation*}
Since there are at most $k$ colours, this process most terminate with some $j \leq k$. Thus $\alpha_i \geq 2^{-2^{O(k)}}$ for all $i\leq j$, and hence $r_i \leq 2^{2^{2^{O(k)}}}$ for all $i\leq j$.  But then $d_{i+1} = d_ir_i$ and hence $d_i \leq   2^{2^{2^{O(k)}}}$ for all $i\leq j$.  Finally $N_{i+1} \geq N_i^{1/4}/d_{i+1}^2$ for all $i\leq j$ and so $N_j \geq 2^{-2^{2^{O(k)}}}N^{4^{-k}}$. However at stage $j$ we have $N_j\leq \exp(\alpha_j^{-O(1)}) \leq 2^{2^{2^{O(k)}}}$, and the bound on $N$ follows.
\end{proof}

\section{Counting in the Furstenberg-S{\'a}rk{\" o}zy Theorem}\label{sec.fs}

Our aim now is to prove the following.
\begin{corollary*}[Corollary \ref{cor.ct}]
Suppose that $A \subset [N]$ has size at least $\alpha N$.  Then either $N \leq \exp(\alpha^{-O(1)})$ or there are natural numbers $r \leq \exp(\alpha^{-O(1)})$ and $L \geq N^{\frac{1}{4}}$  such that
\begin{equation*}
\#\{x \in r \cdot [L]: x^2 \in A-A\} \geq \frac{1}{2}\alpha L.
\end{equation*}
\end{corollary*}
This can be read out of S{\'a}rk{\" o}zy's original argument in \cite{sar::0}.  Unfortunately that argument is presented to deal with existence rather than counting so we have to go some distance into the proof to extract what we need.

We shall make use of the Fourier transform: write $\T:=\R/\Z$, and for $\theta \in \T$ put $e(\theta):=\exp(2\pi i \theta )$.  Given $f \in \ell_1(\Z)$ we write
\begin{equation*}
\wh{f}(\theta):=\sum_{n \in \Z}{f(n)\overline{e(\theta n)}} \text{ for all }\theta \in \T,
\end{equation*}
and $\wt{f}(z):=\overline{f(-z)}$.  Finally, for $f,g \in \ell_1(\Z)$ we write
\begin{equation*}
f \ast g(z):=\sum_{x+y=z}{f(x)g(y)} \text{ for all }z \in \Z.
\end{equation*}
For a finite non-empty set of integers $S$ we write $m_S$ for the function assigning mass $\#S^{-1}$ to each element $S$ and $0$ elsewhere.

We shall actually prove the following from which Corollary \ref{cor.ct} follows immediately (with the $L$ in the corollary being the $\sqrt{L'}$ of the proposition) since $\sum_{x-y=u}{1_{A'}(x)1_{A'}(y)} \leq \#A'$.
\begin{proposition}\label{prop.itkey}
Suppose that $A \subset [N]$ has size at least $\alpha N$.  Then there are natural numbers $r \leq \exp(\alpha^{-O(1)})$, $L \geq (\log N)^{\alpha^{-O(1)}}N$ and $L'\geq \alpha^{O(1)}L$  such that $A':=A \cap (x_0+r^2\cdot [L])$ has $\alpha':=\#A'/L$ with $\alpha' \geq \alpha $ and 
\begin{equation*}
\sum_{x-y=z^2}{1_{A'}(x)1_{A'}(y)1_{r\cdot [\sqrt{L'}]}(z)} \geq \frac{1}{2}(\alpha')^2LL'.
\end{equation*}
\end{proposition}
The argument is an iteration of the following standard exercise in the circle method in which we can afford to be far sloppier than S{\'a}rk{\" o}zy.
\begin{lemma}\label{lem.it}
Suppose that $A \subset [N]$ has size $\alpha N$, and $N'$ is a further integer.  Then at least one of the following holds:
\begin{enumerate}
\item\label{pt1} $N' \geq \alpha^{O(1)}N$;
\item\label{pt2} 
\begin{equation*}
\sum_{x-y=z^2}{1_{A}(x)1_{A}(y)1_{[\sqrt{N'}]}(z)} \geq \frac{1}{2}\alpha^2N\sqrt{N'};
\end{equation*}
\item\label{pt3} there is some $q \leq \alpha^{-O(1)}$ and $N''\geq \alpha^{O(1)}N'/\log N'$ such that 
\begin{equation*}
\sup_x{\#A\cap (x+q^2\cdot [N''])} \geq (\alpha + \alpha^{O(1)})N''.
\end{equation*}
\end{enumerate}
\end{lemma}
\begin{proof}
Write $f:=1_A - \alpha 1_{[N]}$ and $S:=\{1 \leq z^2 \leq N'\}$, and note that
\begin{equation*}
\langle f \ast \wt{f},1_S \rangle_{\ell_2} = \sum_{x-y=w}{1_A(x)1_{A}(y)1_S(w)} -\alpha^2N\sqrt{N'} +O(N')^{3/2},
\end{equation*}
whence either $N' \geq \alpha^{O(1)}N$ or
\begin{equation*}
\int{|\wh{f}(\theta)|^2|\wh{1_S}(\theta)|d\theta} \geq \frac{1}{4}\alpha^2N\sqrt{N'}.
\end{equation*}
Let $Q:=c\alpha^2 N'/\log N'$ for an absolute $c>0$ to be chosen shortly.  By the box principle, for every $\theta \in \T$ there are integers $a$ and $q$ with $(a,q)=1$, $1\leq q \leq Q$ and such that $|\theta -a/q| \leq 1/qQ$.  \cite[Lemma 4]{sar::0} (the notation for our $\wh{1_S}$ is S{\'a}rk{\" o}zy's $T$, defined on \cite[p126]{sar::0}) is Weyl's inequality in a form we can make easy use of:
\begin{equation}\label{eqn.bd}
|\wh{1_S}(\theta)| = O(\sqrt{N'/q} + (\sqrt{N'}\log q)^{\frac{1}{2}} + (q\log q)^{1/2}) = O(1/\sqrt{q} + c\alpha) \sqrt{N'}.
\end{equation}
We define the major arcs to be
\begin{equation*}
\mathfrak{M}_q:=\left\{\theta \in \T: \exists a\text{ with }(a,q)=1 \text{ and }  \left|\theta -\frac{a}{q}\right| \leq \frac{1}{qQ}\right\},
\end{equation*}
so that by Parseval's inequality and (\ref{eqn.bd}) with $Q_0:=(c\alpha)^{-2}$ we have
\begin{equation*}
\int_{\bigcup_{q=Q_0}^Q{\mathfrak{M}_q}}{|\wh{f}(\theta)|^2|\wh{1_S}(\theta)|d\theta} = O(c\alpha^2N\sqrt{N'}).
\end{equation*}
Thus there is an absolute $c>0$ such that 
\begin{equation*}
\sum_{q =1}^{Q_0}{O\left(\frac{1}{\sqrt{q}}\right)\int_{\mathfrak{M}_q}{|\wh{f}(\theta)|^2d\theta}}\geq \frac{1}{8}\alpha^2N.
\end{equation*}
We conclude that there is some $q =O(\alpha^{-2})$ such that
\begin{equation*}
\int_{\mathfrak{M}_q}{|\wh{f}(\theta)|^2d\theta}=\Omega(\alpha^3N) .
\end{equation*}
Let $P$ be a progression of length $2Q+1$ and common difference $q^2$.  Then for all $\theta \in \mathfrak{M}_q$ we have
\begin{equation*}
|\wh{m_P}(\theta)| =\left|\frac{\sin(Q\pi q^2\theta)}{\#P\sin (\pi q^2 \theta)}\right| =\Omega(q^{-1}).
\end{equation*}
By Parseval's theorem we conclude that
\begin{equation*}
\|(1_A - \alpha 1_{[N]})\ast m_P\|_{\ell_2}^2 =\int{|\wh{f}(\theta)|^2|\wh{m_P}(\theta)|^2d\theta} \geq \alpha^{O(1)}N
\end{equation*}
and so
\begin{equation*}
\|1_A \ast m_{P'}\|_{\ell_2}^2  \geq \alpha^2N+ \alpha^{O(1)}N- O(\#P'/N),
\end{equation*}
and this gives the result with $N''=\#P'$ by averaging.
\end{proof}

\begin{proof}[Proof of Proposition \ref{prop.itkey}] We construct natural numbers $d_i$, $N_i$, and $N_i'$ iteratively with $N_i' = \alpha^{O(1)}N_i$ such that conclusion (\ref{pt1}) of Lemma \ref{lem.it} does not hold if $N_i'$ is the `further integer' of that lemma.  At stage $i$ let $x_i$ be such that $\#A\cap(x_i+d_i^2\cdot [N_i])$ is maximal over all possible choices; write $\alpha_i$ for the ratio of this size to $N_i$.  

Apply Lemma \ref{lem.it} to the set $S_i:=\{x\in [N_i]: x_i+d_i^2x \in A\}$ with the `further integer' being $N_i'$.  In case (\ref{pt2}) stop; in case (\ref{pt3}) there is some $q_i \leq \alpha_i^{-O(1)}$ and $N_{i+1} \geq \alpha^{O(1)}N_i'/\log N_i'$ such that
\begin{equation*}
 \alpha_i + \alpha_i^{O(1)} \leq \frac{\#S_i \cap (x+q_i^2\cdot [N_{i+1}])}{N_{i+1}}=\frac{\#A \cap (x_i+d_i^2x+(d_i^2q_i^2)\cdot [N_{i+1}])}{N_{i+1}}.
\end{equation*}
Set $d_{i+1}=d_iq_i$ so that $\alpha_{i+1} \geq \alpha_i + \alpha_i^{O(1)}$.  This process terminates for some $j=\alpha^{-O(1)}$ because density cannot exceed $1$.  We set $L:=N_j$, $L':=N_j'$, $r:=d_j$ and have the result since $N_j' \geq \alpha^{O(1)}N_j$, and $N_{i+1} \geq \alpha^{O(1)}N_i/\log N_i$ and $d_{i+1} \leq d_i\alpha^{-O(1)}$ for all $i<j$.
\end{proof}

\section*{Acknowledgement} The author should like to thank the referee for a careful reading of the paper; the editors of the volume for the invitation to submit; and most importantly Endre Szemer{\'e}di for many years of support and interesting discussions.

\bibliographystyle{halpha}

\bibliography{references}

\end{document}